\documentclass[12pt]{article}
\usepackage{ae} 
\usepackage{aecompl} 

\usepackage{amsmath}
\usepackage{amsthm}
\usepackage{amsfonts}
\usepackage{amssymb}
\usepackage{stmaryrd}
\usepackage{mathtools}
\usepackage{bm}
\usepackage{bbm}
\usepackage{tikz}
\usetikzlibrary{calc}
\usetikzlibrary{arrows}

\usepackage[%
  pdftex,
  plainpages=false,
  pdfpagelabels,
  colorlinks,
  citecolor=black,
  linkcolor=black,
  urlcolor=black,
  filecolor=black,
  bookmarksopen=false
]{hyperref} 
\usepackage[all]{hypcap} 

\def\({\left(}
\def\){\right)}
\def\[{\left[}
\def\]{\right]}
\def\<{\left\langle}
\def\>{\right\rangle}
\def\lv{\left\lvert}

\def\rv{\right\rvert}

\def\llb{\left\llbracket}
\def\rrb{\right\rrbracket}

\def\iff{\Longleftrightarrow}

\let\geq\geqslant
\let\leq\leqslant
\let\:\colon

\let\epsilon\varepsilon
\let\oldphi\phi
\let\phi\varphi

\def\NN{\ensuremath{\mathbb{N}}}

\def\RR{\ensuremath{\mathbb{R}}}

\def\One{\ensuremath{\mathbbm{1}}}

\def\cA{\ensuremath{\mathcal{A}}}

\def\cF{\ensuremath{\mathcal{F}}}

\let\phi\oldphi

\DeclareMathOperator{\Tr}{Tr}
\DeclareMathOperator{\Aut}{Aut}
\newcommand{\df}{\stackrel{\rm def}{=}}
\newcommand{\edge}[2]{\langle #1,#2 \rangle}
\DeclareMathOperator{\Hom}{Hom}
\newcommand{\of}[1]{\left( #1 \right)}
\newcommand{\expect}[1]{ {\mathbf E}\! \left[ #1 \right] }
\newcommand{\function}[2]{\colon #1 \longrightarrow #2}
\newcommand{\set}[2]{\left\{\hspace{0.3ex} #1 \middle| \hspace{0.2ex} #2 \right\}}
\newcommand{\eval}[2]{\llbracket #1 \rrbracket_{#2}}
\newcommand{\phirdm}{\phi_{\text{qr}}}
\newcommand{\as}{\;\;\text{a.s.}}
\renewcommand{\One}{1}

\newtheoremstyle{plain}
  {\medskipamount}
  {\smallskipamount}
  {\slshape}
  {0pt}
  {\bfseries}
  {.}
  { }
  {\thmname{#1}\thmnumber{ #2}{\normalfont\thmnote{ (#3)}}}

\theoremstyle{plain}

\newtheorem{theorem}{Theorem}[section]

\newtheorem*{remark*}{Remark}
\newtheorem*{example*}{Example}

\title{On the Density of Transitive Tournaments}
\author{Leonardo Nagami Coregliano\thanks{
Instituto de Matem\'atica e Estat\'istica, Universidade de S\~ao Paulo, \nolinkurl{lenacore@ime.usp.br}. Work done while visiting
University of Chicago, supported by Funda\c c\~ao de Amparo \`a Pesquisa do Estado de S\~ao Paulo (FAPESP) under grants no.~2013/23720-9 and~2014/15134-5.}\and Alexander A.~Razborov\thanks{University of Chicago, \nolinkurl{razborov@cs.uchicago.edu}. Part of this work was done while the
author was at
Steklov Mathematical Institute, supported by the Russian Foundation
for Basic Research, and at Toyota Technological Institute, Chicago.}}
\begin{document}
\maketitle

\begin{abstract}
We prove that for every fixed~$k$, the number of occurrences of the transitive tournament~$\Tr_k$ of order~$k$ in a tournament~$T_n$ on~$n$ vertices is asymptotically minimized when~$T_n$ is random. In the opposite direction, we show that any sequence of tournaments~$\{T_n\}$ achieving this minimum for any fixed~$k\geq 4$ is necessarily quasi-random. We present several other characterizations of quasi-random
tournaments nicely complementing previously known results and relatively easily following from our proof techniques.
\end{abstract}

For a fixed combinatorial object~$M$, e.g.~a graph, a hypergraph or a tournament, the number of its occurrences in a large random object~$N$ of the same signature is always easy to compute asymptotically. Reverse problems have become the subject of active research in combinatorics.

In the first direction along these lines, we are interested in the questions of the kind ``for which templates~$M$ the number of occurrences is asymptotically optimized by the random object'' or, in other terms, when the latter forms an extremal configuration for the corresponding extremal problem. Perhaps, one of the most prominent settings in which this question has been studied systematically is that of~\emph{$k$-common graphs}~\cite{JST}. In this setting, $M$ is an ordinary graph, $N$ is a~$k$-edge-coloring of a complete graph, and an ``occurrence'' means a monochromatic copy. The~$k$-common graphs are those for which the random edge coloring asymptotically minimizes the number of occurrences, and they were studied in many papers, see e.g.~\cite{Erd5,BuR,Sid2,JST,Tho,CuY,wheel}.

The second direction focusses on the complementary question: ``When all extremal configurations for this extremal problem are necessarily quasi-random?'' We do not attempt to review here the theory of quasi-randomness: beginning from the seminal papers by Thomason~\cite{Tho2} and
Chung, Graham, Wilson~\cite{CGW}, it has developed into a vast field surveyed for example in~\cite{KrS}. But the main thrust of this theory is that many a priori different properties eventually lead to the same class of combinatorial objects, hereafter called {\em quasi-random}. Perhaps, the most intuitive of these properties says that quasi-random objects are precisely those~$N$ in which {\em all} fixed templates~$M$ have the right density of occurrences, understood asymptotically.

\bigskip

In this note we study questions of both kinds for tournaments. By analogy with graphs, it might be tempting to call a tournament~$T$ {\em common} if for any increasing sequence~$\{T_n\}$ of tournaments we have
\begin{equation}\label{naive}
  \liminf_{n\to\infty} p(T,T_n)\geq \frac{k!}{|\Aut(T)|\cdot 2^{k\choose 2}},
\end{equation}
where~$k\df |V(T)|$, $p(T,T_n)$ is the density with which~$T$ appears in~$T_n$ as an unlabelled sub-tournament and~$\Aut(T)$ denotes the group
of automorphisms of~$T$. In the language of flag algebras~\cite{flag} it can be stated more cleanly and concisely as
\begin{equation} \label{flagmatic}
\forall \phi\in\Hom^+(\mathcal A^0[T_{\text{Tournaments}}], \mathbb R) \of{\phi(T) \geq \frac{k!}{|\Aut(T)|\cdot 2^{k\choose 2}}},
\end{equation}
and in what follows we will mostly use this language. The reader who feels uncomfortable with limit objects should have little difficulties translating our statements to the finite world by analogy with~\eqref{naive}.

Upon a moment's reflection it becomes clear that only transitive tournaments may possibly be common since all others occur with zero density in the increasing sequence~$\{\Tr_n\}$ of transitive tournaments. This gives rise to the natural question if the converse is true, that is if
\begin{equation}\label{main_inequality}
\phi(\Tr_k) \geq \frac{k!}{2^{k\choose 2}},
\end{equation}
where again~$\phi$ is any algebra homomorphism from~$\Hom^+(\mathcal A^0[T_{\text{Tournaments}}], \mathbb R)$. To the best of our knowledge, this question was not explicitly asked before. The cases~$k=0,1,2$ are trivial,
and the case~$k=3$ was implicitly answered by Chung and~Graham in~\cite[Fact~1]{ChGr2}.
With a bit of effort, the bound~\eqref{main_inequality} for~$k=4$ can be extracted from Griffith's paper~\cite[Proposition~3.1~(i)]{Grif}, but his proof does not seem to be generalizable to larger values of~$k$.

Our first main result confirms~\eqref{main_inequality} for all values of~$k$. In other words, transitive tournaments are common and vice versa. This result also has an interesting ``sharp threshold'' flavor to it. Namely, if we take an arbitrarily large transitive tournament and flip even a {\em single} arc in
it\footnote{except for the arcs that when flipped yield transitive tournaments again}, the extremal (minimizing) configuration jumps from the random tournament to the opposite side of the spectrum, i.e.~transitive tournaments.

\medskip

The theory of quasi-random tournaments was inaugurated by Chung and~Graham in~\cite{ChGr2} and followed up in~\cite{KaSa,Grif}. A significant portion of all these papers is devoted to the question formulated above in more general context: for which tournaments~$T$ the equality in~\eqref{flagmatic} implies that~$\phi$ is quasi-random?
We show this for any transitive tournament~$\Tr_k$ with\footnote{The case~$k=4$ can be again extracted from~\cite{Grif}.}~$k\geq 5$.

Many characterizations in~\cite{ChGr2,Grif} are given in terms of ``flag concentration''. Given a tournament~$T$ and an edge~$\edge uv\in E(T)$, all other vertices~$w\in V(T)\setminus \{u,v\}$ can be classified into four classes (``flags''):
\begin{enumerate}
\item $\edge uw, \edge vw\in E(T)$,

\item $\edge wu, \edge wv\in E(T)$,

\item $\edge uw, \edge wv\in E(T)$.

\item $\edge vw, \edge wu\in E(T)$,
\end{enumerate}

Following and expanding a bit the notation in~\cite{ch}, we let~$O^A(u,v)$, $I^A(u,v)$, $\Tr_3^A(u,v)$ and~$\vec C_3^A(u,v)$ denote the numbers of vertices in the four classes
(taken in this order, see also Figure~\ref{fig:typesandflags} below) divided by~$|V(T)|-2$. One of the keystone characterizations in~\cite{ChGr2} ($P_4$, to be exact) says that a sequence of tournaments is quasi-random if and only if\footnote{This again can be more cleanly stated in the language of flag algebras~\cite[\S 3.2]{flag} (which we do in Section~\ref{sec:flag}) or digraphons~\cite[Example~9.2]{DiJa}.}~$O^A(u,v)$ is ``nearly''~1/4 for ``almost all'' edges~$\edge uv$. A similar statement for~$I^A(u,v)$ follows by duality.

Our methods allow us to do, in exactly the same manner, the two remaining cases: $\Tr_3^A(u,v)$ and~$\vec C_3^A(u,v)$. Stated in the finite language and combined with the previously known results, we now have that a sequence of tournaments~$\{T_n\}$ is quasi-random if and only if
\begin{equation} \label{flag_concentration}
\sum_{\edge uv\in E(T_n)} |F(u,v)-1/4| \leq o(n^2),
\end{equation}
where~$F$ corresponds to {\em any} of the four cases above.

The note is organized as follows. In Section~\ref{sec:flag} we remind some rudimentary concepts from the theory of flag
algebras and show how to treat quasi-randomness in that context. Again, the reader who feels uncomfortable with this
language (and is willing to tolerate a bit of coping with low-order error terms instead) should have no difficulties
replacing~$\sigma$-extensions with averaging over vertices or arcs in large but finite tournaments, see e.g.~\cite{flag_ams}.
In Section~\ref{sec:transitive} we prove the bound~\eqref{flagmatic}, and also that the equality is attained here if
and only if~$\phi$ is quasi-random (Theorem~\ref{thm:extchar}). Finally, in Section~\ref{sec:Aflags} we prove additional
characterizations of quasi-random tournaments in terms of ``flag concentrations'' (Theorems~\ref{thm:chartr3A}
and~\ref{thm:charc3A}).

\section{Quasi-Randomness in Flag Algebras}
\label{sec:flag}

In this section, we translate the results of quasi-randomness that we are going to use to the language of flag algebras. We assume the reader has some familiarity with not only the basic setting of flag algebra, but also with extensions of homomorphisms~\cite[\S 3.2]{flag}.

Following the notation of~\cite{flag,ch}, we consider the theory of tournaments~$T_{\text{Tournaments}}$ and we let~$1$ denote the (unique) type of size~$1$ and~$A$ denote the type of size~$2$ such that the vertex labelled with~$1$ beats
the other (labelled) vertex (see Figure~\ref{fig:typesandflags}). For a type~$\sigma$, we denote the unity of the
algebra~$\cA^\sigma$ by~$\One_\sigma$, and, as always, $\One_0$ is abbreviated to~$\One$.

\begin{figure}[ht]
  \begin{center}
    \begin{tikzpicture}[scale=1.1]
  \foreach \i in {0,...,3}{
    \pgfmathsetmacro{\base}{2*\i}
    \pgfmathsetmacro{\next}{\base+1}

    \coordinate (T3A\i) at (\base cm, 1cm);
    \coordinate (T3B\i) at (\next cm, 1cm);
    \coordinate (T3C\i) at ($1/2*(T3A\i) + 1/2*(T3B\i) + (0cm,0.707106cm)$);

    \coordinate (T4A\i) at (\base cm,-1.5cm);
    \coordinate (T4B\i) at (\next cm,-1.5cm);
    \coordinate (T4C\i) at (\next cm,-0.5cm);
    \coordinate (T4D\i) at (\base cm,-0.5cm);

    \coordinate (TpA\i) at (\base cm, -3cm);
    \coordinate (TpC\i) at (\next cm, -3cm);
    \coordinate (TpB\i) at ($1/2*(TpA\i) + 1/2*(TpC\i)$);

    \coordinate (F2A\i) at (\base cm,-4.5cm);
    \coordinate (F2B\i) at (\next cm,-4.5cm);

    \coordinate (F3A\i) at (\base cm, -6cm);
    \coordinate (F3B\i) at (\next cm, -6cm);
    \coordinate (F3C\i) at ($1/2*(F3A\i) + 1/2*(F3B\i) + (0cm,0.707106cm)$);

    \coordinate (F4A\i) at (\base cm,-8.5cm);
    \coordinate (F4B\i) at (\next cm,-8.5cm);
    \coordinate (F4C\i) at (\next cm,-7.5cm);
    \coordinate (F4D\i) at (\base cm,-7.5cm);
  }

  \foreach \i in {1,2}{
    \filldraw (T3A\i) circle (1pt);
    \filldraw (T3B\i) circle (1pt);
    \filldraw (T3C\i) circle (1pt);

    \draw[thick, arrows={-latex}] (T3A\i) -- (T3B\i);
    \draw[thick, arrows={-latex}] (T3B\i) -- (T3C\i);
  }
  \draw[thick, arrows={-latex}] (T3A1) -- (T3C1);
  \draw[thick, arrows={-latex}] (T3C2) -- (T3A2);

  \node[below] at ($1/2*(T3A1) + 1/2*(T3B1)$) {$\Tr_3$};
  \node[below] at ($1/2*(T3A2) + 1/2*(T3B2)$) {$\vec C_3$};
  
  \foreach \i in {0,...,3}{
    \filldraw (T4A\i) circle (1pt);
    \filldraw (T4B\i) circle (1pt);
    \filldraw (T4C\i) circle (1pt);
    \filldraw (T4D\i) circle (1pt);

    \draw[thick, arrows={-latex}] (T4A\i) -- (T4B\i);
    \draw[thick, arrows={-latex}] (T4B\i) -- (T4C\i);
  }
  \draw[thick, arrows={-latex}] (T4A0) -- (T4C0);
  \draw[thick, arrows={latex-}] (T4A1) -- (T4C1);
  \draw[thick, arrows={latex-}] (T4A2) -- (T4C2);
  \draw[thick, arrows={latex-}] (T4A3) -- (T4C3);

  \draw[thick, arrows={-latex}] (T4A0) -- (T4D0);
  \draw[thick, arrows={-latex}] (T4B0) -- (T4D0);
  \draw[thick, arrows={-latex}] (T4C0) -- (T4D0);

  \draw[thick, arrows={latex-}] (T4A1) -- (T4D1);
  \draw[thick, arrows={latex-}] (T4B1) -- (T4D1);
  \draw[thick, arrows={latex-}] (T4C1) -- (T4D1);

  \draw[thick, arrows={-latex}] (T4A2) -- (T4D2);
  \draw[thick, arrows={-latex}] (T4B2) -- (T4D2);
  \draw[thick, arrows={-latex}] (T4C2) -- (T4D2);

  \draw[thick, arrows={-latex}] (T4A3) -- (T4D3);
  \draw[thick, arrows={-latex}] (T4B3) -- (T4D3);
  \draw[thick, arrows={latex-}] (T4C3) -- (T4D3);

  \node[below] at ($1/2*(T4A0) + 1/2*(T4B0)$) {$\Tr_4$};
  \node[below] at ($1/2*(T4A1) + 1/2*(T4B1)$) {$W_4$};
  \node[below] at ($1/2*(T4A2) + 1/2*(T4B2)$) {$L_4$};
  \node[below] at ($1/2*(T4A3) + 1/2*(T4B3)$) {$R_4$};

  \filldraw (TpB1) circle (1pt);
  \filldraw (TpA2) circle (1pt);
  \filldraw (TpC2) circle (1pt);

  \draw[thick, arrows={-latex}] (TpA2) -- (TpC2);

  \node[left] at (TpB1) {$1$};
  \node[left] at (TpA2) {$1$};
  \node[right] at (TpC2) {$2$};

  \node[below] at (TpB1) {$1$};
  \node[below] at (TpB2) {$A$};

  \filldraw (F2B1) circle (1pt);
  \filldraw (F2A2) circle (1pt);

  \draw[thick, arrows={-latex}] (F2B1) -- (F2A2);

  \node[left] at (F2B1) {$1$};

  \node[below] at ($1/2*(F2B1) + 1/2*(F2A2)$) {$\alpha$};

  \foreach \i in {0,...,3}{
    \filldraw (F3A\i) circle (1pt);
    \filldraw (F3B\i) circle (1pt);
    \filldraw (F3C\i) circle (1pt);

    \draw[thick, arrows={-latex}] (F3A\i) -- (F3B\i);

    \node[left] at (F3A\i) {$1$};
    \node[right] at (F3B\i) {$2$};
  }
  \draw[thick, arrows={-latex}] (F3A0) -- (F3C0);
  \draw[thick, arrows={-latex}] (F3B0) -- (F3C0);

  \draw[thick, arrows={latex-}] (F3A1) -- (F3C1);
  \draw[thick, arrows={latex-}] (F3B1) -- (F3C1);

  \draw[thick, arrows={-latex}] (F3A2) -- (F3C2);
  \draw[thick, arrows={latex-}] (F3B2) -- (F3C2);

  \draw[thick, arrows={latex-}] (F3A3) -- (F3C3);
  \draw[thick, arrows={-latex}] (F3B3) -- (F3C3);

  \node[below] at ($1/2*(F3A0) + 1/2*(F3B0)$) {$O^A$};
  \node[below] at ($1/2*(F3A1) + 1/2*(F3B1)$) {$I^A$};
  \node[below] at ($1/2*(F3A2) + 1/2*(F3B2)$) {$\Tr_3^A$};
  \node[below] at ($1/2*(F3A3) + 1/2*(F3B3)$) {$\vec C_3^A$};

  \filldraw (F4B1) circle (1pt);
  \filldraw (F4C1) circle (1pt);
  \filldraw (F4A2) circle (1pt);
  \filldraw (F4D2) circle (1pt);

  \draw[thick, arrows={-latex}] (F4B1) -- (F4A2);
  \draw[thick, arrows={-latex}] (F4D2) -- (F4A2);
  \draw[thick, arrows={-latex}] (F4B1) -- (F4C1);
  \draw[thick, arrows={-latex}] (F4C1) -- (F4D2);
  \draw[thick, arrows={-latex}] (F4B1) -- (F4D2);
  \draw[thick, arrows={-latex}] (F4C1) -- (F4A2);

  \node[left] at (F4B1) {$1$};
  \node[right] at (F4A2) {$2$};
  \node[below] at ($1/2*(F4B1) + 1/2*(F4A2)$) {$\Tr_4^A$};
\end{tikzpicture}
    \caption{Types and flags used.}
    \label{fig:typesandflags}
  \end{center}
\end{figure}
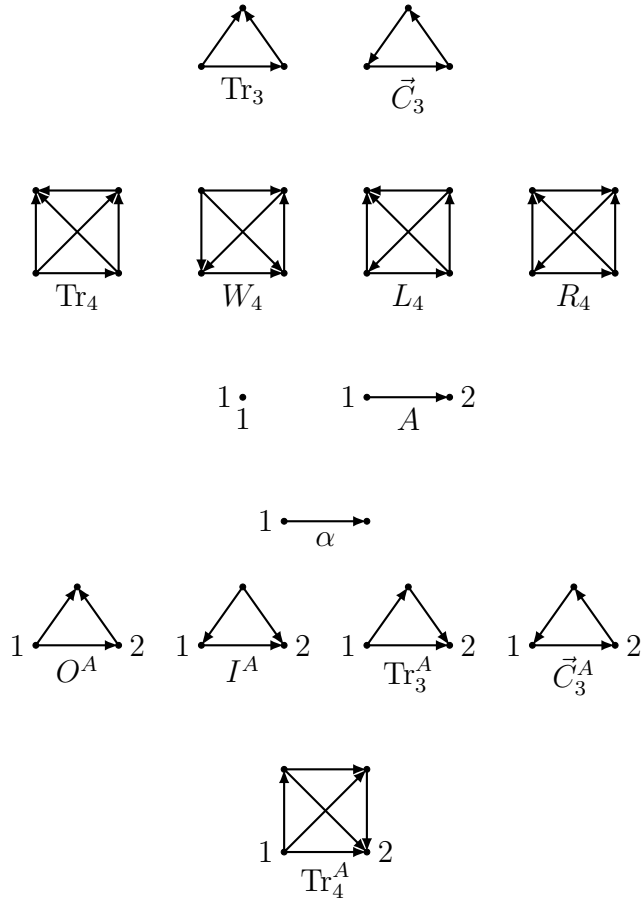

We have already introduced the notation~$\Tr_k$ to denote the transitive tournament of size~$k$. We let~$\vec C_3$ be the~$3$-directed cycle (that is, the only other tournament of size~$3$) and we define the following tournaments of size~$4$.
\begin{enumerate}
\item The tournament~$W_4$, which is the (unique) \emph{non-transitive} tournament of size~$4$ that has a vertex with outdegree~$3$ (that is, there is a ``winner'' in~$W_4$);
\item The tournament~$L_4$, which is the (unique) \emph{non-transitive} tournament of size~$4$ that has a vertex with indegree~$3$ (that is, there is a ``loser'' in~$L_4$);
\item The tournament~$R_4$, which is the (unique) tournament of size~$4$ that has outdegree sequence~$(1,1,2,2)$.
\end{enumerate}
Note that along with~$\Tr_4$, this list covers all tournaments of size~$4$.

We define the~$1$-flag~$\alpha$ as the (unique)~$1$-flag of size~$2$ in which the labelled vertex beats the unlabelled vertex. We also define the following~$A$-flags of size~$3$.
\begin{enumerate}
\item The flag~$O^A$, in which the only unlabelled vertex is beaten by both labelled vertices;
\item The flag~$I^A$, in which the only unlabelled vertex beats both labelled vertices;
\item The flag~$\Tr_3^A$, which is the only remaining~$A$-flag whose underlying model is~$\Tr_3$;
\item The flag~$\vec C_3^A$, which is the only~$A$-flag whose underlying model is~$\vec C_3$.
\end{enumerate}
Again, this is the complete list of~$A$-flags of size~$3$.

We also follow the original notation of flag algebras when using the downward operator~$\llb{}\cdot{}\rrb_\sigma$ to
the~$0$-algebra or when using~$\sigma$-extensions of homomorphisms~$\phi\in\Hom^+(\cA^0,\RR)$
(which are denoted by~$\bm{\phi^\sigma}$). We remind that~$\bm{\phi^\sigma}$ can be conveniently viewed~\cite[Definition~10]{flag} as
the unique~$\Hom^+(\cA^\sigma,\RR)$-valued random variables satisfying  the identities
\begin{equation} \label{extensions}
\expect{\bm{\phi^\sigma}(F)} = \frac{\phi(\llb F\rrb_\sigma)}{\phi(\llbracket\One_\sigma\rrbracket_\sigma)}
\end{equation}
for every~$F\in\cF^\sigma$.

Finally, for a flag~$F_0$ of type~$\sigma$ with~$|\sigma|+1$ vertices we have the algebra homomorphism~$\pi^{F_0}\function{\mathcal A^0}{\mathcal A^\sigma_{F_0}}$, where~$\mathcal A^\sigma_{F_0}$ is the localization of the algebra~$\mathcal A^\sigma$ with respect to the multiplicative system~$\set{F_0^\ell}{\ell\in \mathbb N}$~\cite[\S 2.3.2]{flag}. Intuitively, it corresponds to taking the sub-model induced by the flag~$F_0$, and the
localization is needed for proper normalization resulting from decreasing the set of vertices. In this note we will only need the operator~$\pi^{O^A}$.

We will not need these concepts in the more complicated scenario when the smaller type is also non-trivial.

\smallskip
Let us denote the homomorphism of~$\Hom^+(\cA^0,\RR)$ corresponding to the random tournament by~$\phirdm$, that is, it is the almost sure limit of the sequence of random tournaments (where each arc orientation is picked independently with probability~$1/2$) when the number of vertices goes to infinity.

Note that for every tournament~$T$, we have
\begin{align*}
  \phirdm(T) & = \frac{\lv V(T)\rv!}{\lv\Aut(T)\rv\cdot 2^{\binom{|T|}{2}}}.
\end{align*}

In particular, since transitive tournaments possess only trivial automorphisms, we have
\begin{align*}
  \phirdm(\Tr_k) & = \frac{k!}{2^{\binom{k}{2}}},
\end{align*}
for every~$k\in\NN$.

As for the quasi-randomness properties (the~$P$ properties of Chung--Graham~\cite{ChGr2}), we are interested in the following ones.
\begin{itemize}
\item $P_2$: $\phi(\Tr_4 + R_4) = 3/4$;
\item $P_4$: $\bm{\phi^A}(O^A) = 1/4$ a.s.
\end{itemize}
That is, the above are equivalent to each other, and are equivalent to the fact~$\phi=\phirdm$ (which, in the terminology of~\cite{ChGr2}, is precisely~$\forall s\in\NN, P_1(s)$).

The translation of property~$P_2$ follows from simple arithmetics relating homomorphism densities to flag algebra densities, and the translation of~$P_4$ follows from the definition of extensions of homomorphisms.

\section{The Minimum Density of Transitive Tournaments} \label{sec:transitive}

We start by defining for every~$k\geq 2$ the~$1$-flag~$\Tr_k^W$ as the flag obtained from~$\Tr_k$ by labelling its winner (i.e.~the unique vertex with outdegree~$k-1$) and the~$A$-flag~$\Tr_k^{W2}$ as the flag obtained from~$\Tr_k$ by labelling its winner with the label~$1$ and its runner-up (i.e.~the unique vertex with outdegree~$k-2$) with the label~$2$. In particular, with this definition we have~$\alpha=\Tr_2^W$ and~$O^A=\Tr_3^{W2}$.

We further note that~$\alpha^2 = \Tr_3^W$, $\llb\alpha\rrb_1 = \Tr_2/2 = \One/2$ and for every~$k\geq 2$, we have
\begin{align*}
  \Tr_k & = k\llbracket\Tr_k^W\rrbracket_1 = k(k-1)\llbracket\Tr_k^{W2}\rrbracket_A;\\
  \Tr_k^{W2} & = \pi^{O^A}(\Tr_{k-2})\cdot(O^A)^{k-2}.
\end{align*}

We are now ready to prove the main result. We remind the reader that, although the proof is presented for all~$k\in\NN$, the result for~$k\leq 4$ was known before.

\begin{theorem}\label{thm:extchar}
  In the theory of tournaments, for every~$k\in\NN$ and every~$\phi\in\Hom^+(\cA^0,\RR)$, we have
  \begin{align*}
    \phi(\Tr_k) & \geq \frac{k!}{2^{\binom{k}{2}}}.
  \end{align*}

  Furthermore, for any fixed~$k\geq 4$, the equality holds if and only if~$\phi$ is the quasi-random homomorphism~$\phirdm$.
\end{theorem}

\begin{proof}
  For~$k=0,1,2$, the result is trivial, because as flag algebra elements, we have~$\Tr_0=\Tr_1=\Tr_2=\One$ and the right hand side of the formula evaluates to~$1$ in these cases.

  For~$k=3$, note that since~$\llbracket\One_1\rrbracket_1 = \One$, we have
  $$
    \phi(\Tr_3) = 3\phi(\llbracket\alpha^2\rrbracket_1)
    = 3\expect{\bm{\phi^1}(\alpha)^2}
     \geq
    3\expect{\bm{\phi^1}(\alpha)}^2
    = 3\phi(\llb\alpha\rrb_ 1)^2
    = \frac{3}{4}
    = \frac{3!}{2^{\binom{3}{2}}},
  $$

  Now we proceed by induction. Suppose that~$k\geq 4$ and that the result is valid for~$k-2$.

  Since~$\llbracket\One_A\rrbracket_A = \One/2$, we have
  \begin{align*}
    \phi(\Tr_k) & =
    k(k-1)\phi(\llbracket\pi^{O^A}(\Tr_{k-2})\cdot(O^A)^{k-2}\rrbracket_A)
    \\
    & = \frac{k(k-1)}{2} \expect{
    \bm{\phi^A}(\pi^{O^A}(\Tr_{k-2}))
    \bm{\phi^A}\of{(O^A)^{k-2}}}.
  \end{align*}

  Since~$\bm{\phi^A}(O^A)> 0$ implies that~$\phi^A\circ \pi^{O^A}\in\Hom^+(\cA^0,\RR)$, by inductive hypothesis we have
  $$
    \expect{
    \bm{\phi^A}(\pi^{O^A}(\Tr_{k-2}))
    \bm{\phi^A}\of{(O^A)^{k-2}}}
    \geq
    \frac{(k-2)!}{2^{\binom{k-2}{2}}}
    \expect{\bm{\phi^A}((O^A)^{k-2})}.
  $$

  By Jensen's inequality, we have
  \begin{equation} \label{jensen}
  \expect{\bm{\phi^A}(O^A)^{k-2}}
   \geq
    \expect{\bm{\phi^A}(O^A)}^{k-2}
    = (2\phi(\llbracket O^A\rrbracket_A))^{k-2}
    = \(\frac{\phi(\Tr_3)}{3}\)^{k-2}.
  \end{equation}

  Since we already proved the case~$k=3$, we have~$\phi(\Tr_3)\geq 3/4$, so putting things
  together, we get
  \begin{equation} \label{chain}
    \phi(\Tr_k) \geq
    \frac{k(k-1)}{2}\cdot\frac{(k-2)!}{2^{\binom{k-2}{2}}}
    \(\frac{\phi(\Tr_3)}{3}\)^{k-2}
    \geq \frac{k!}{2^{\binom{k}{2}}}.
  \end{equation}

  \medskip

  For the ``furthermore'' part, note that if equality holds for~$\phi$, then along the proof we have equalities instead of inequalities.

  In particular, since Jensen's inequality was used in~\eqref{jensen} with the function~$x^{k-2}$, which for~$k\geq 4$ is strictly convex\footnote{This is precisely the reason why the second part of Theorem~\ref{thm:extchar} does not hold for~$k=3$.} on~$[0,+\infty)$, we have that~$\bm{\phi^A}(O^A)$ is a.s.~a constant, denote it by~$C$.
Furthermore, we also have~$\phi(\Tr_3) = 3/4$ as it is used in the chain of inequalities~\eqref{chain}. This allows us to compute~$C$ as follows:
  $$
    C=\expect{\bm{\phi^A}(O^A)} =
    2\phi(\llbracket O^A\rrbracket_A)
    = \frac{\phi(\Tr_3)}{3}
    = \frac{1}{4}.
  $$

  Therefore~$\phi$ satisfies~$P_4$, hence~$\phi=\phirdm$.
\end{proof}

\begin{remark*}
  Note that if we wanted to prove just the inequality statement (i.e.~without the ``furthermore'' part), we could have done the induction with a simpler argument involving~$\pi^\alpha(\Tr_{k-1})\alpha^{k-1}$ instead of~$\pi^{O^A}(\Tr_{k-2})\cdot(O^A)^{k-2}$.
  Moreover, modulo the following (straightforward) generalization of Cauchy--Schwarz inequality~\cite[(22)]{flag}:
  $$
  \eval{f^k}{\sigma} \cdot \eval{\One_\sigma}{\sigma}^{k-1} \geq \eval{f}{\sigma}^k
  $$
  we could have also avoided the extension~$\bm{\phi^A}$ at all and reduced the whole argument to a several-lines calculation.
\end{remark*}

\section{New Characterizations with~$A$-Flags and~$A$-Extensions} \label{sec:Aflags}

As we already said in the introduction, Chung and~Graham presented in~\cite{ChGr2} quasi-random characterizations involving the~$A$-flag~$O^A$. Reverting all arcs of the tournament, we get the dual characterization involving~$I^A$ (which is stated below for completeness).\footnote{Just reverting the arcs, we actually get a characterization involving flags of the other type of size~$2$, but then we may use the flag algebra isomorphism that swaps the labels~$1$ and~$2$ to arrive at the desired characterization.}
\begin{itemize}
\item $P_4'$: $\bm{\phi^A}(I^A) = 1/4$ a.s.
\end{itemize}

In this section, we prove that the two remaining~$A$-flags~$\Tr_3^A$ and~$\vec C_3^A$ of size~$3$ also characterize quasi-randomness of tournaments. Note, however, that these flags are self-dual, so they must be treated separately.

\begin{theorem}\label{thm:chartr3A}
  In the theory of tournaments, for every~$\phi\in\Hom^+(\cA^0,\RR)$, we have that
  \begin{align*}
    \bm{\phi^A}(\Tr_3^A) = \frac{1}{4} \as
  \end{align*}
  if and only if~$\phi$ is the quasi-random homomorphism~$\phirdm$.
\end{theorem}

\begin{proof}
This is similar to the second part of the proof of Theorem~\ref{thm:extchar}, except that now we use the flag~$\Tr_3^A$ instead of~$O^A$.
  Let~$\Tr_4^A$ denote the~$A$-flag obtained from~$\Tr_4$ by labelling its winner with label~$1$ and labelling its loser with label~$2$ (see Figure~\ref{fig:typesandflags}).

  Note that~$(\Tr_3^A)^2 = \Tr_4^A$ and that~$\Tr_4 = 12\llb \Tr_4^A\rrb_A$. Since~$\llbracket\One_A\rrbracket_A = \One/2$, we have
  $$
    \phi(\Tr_4) = 12\phi(\llbracket (\Tr_3^A)^2\rrbracket_A)
    = 6\expect{\bm{\phi^A}(\Tr_3^A)^2}\geq 6\expect{\bm{\phi^A}(\Tr_3^A)}^2
    = 6(2\phi(\llbracket\Tr_3^A\rrbracket_A))^2
    = \frac{2}{3}\phi(\Tr_3)^2.
  $$

  By Theorem~\ref{thm:extchar}, we know that~$\phi(\Tr_3)\geq 3/4$, hence we have
  \begin{align*}
    \phi(\Tr_4) & \geq \frac{3}{8} = \frac{4!}{2^{\binom{4}{2}}}.
  \end{align*}

  But, since~$x^2$ is both strictly convex and strictly increasing on~$[0,+\infty)$, we have that equality holds if and only if we have almost surely
  \begin{align*}
    \bm{\phi^A}(\Tr_3^A) & =
    \expect{\phi^A(\Tr_3^A)}
    = \frac{\phi(\Tr_3)}{3} = \frac{1}{4}.
  \end{align*}

  By the ``furthermore'' part of Theorem~\ref{thm:extchar}, we have
  \begin{align*}
    \bm{\phi^A}(\Tr_3^A) & = \frac{1}{4}\as \iff \phi(\Tr_4) = \frac{4!}{2^{\binom{4}{2}}}
    \iff \phi = \phirdm.
    \qedhere
  \end{align*}
\end{proof}

 For the final flag~$\vec C_3^A$ the proof cannot be done analogously because the underlying tournament is now~$\vec C_3$ instead of~$\Tr_3$. Nevertheless, we are able to obtain the result by other means.

\begin{theorem}\label{thm:charc3A}
  In the theory of tournaments, for every~$\phi\in\Hom^+(\cA^0,\RR)$, we have that
  \begin{align*}
    \bm{\phi^A}(\vec C_3^A) = \frac{1}{4} \as
  \end{align*}
  if and only if~$\phi$ is the quasi-random homomorphism~$\phirdm$.
\end{theorem}

\begin{proof}
In one direction it is simple.
  By Chung and Graham's characterizations~$P_4$ and~$P_4'$ (that is, after reverting the arcs) and by Theorem~\ref{thm:chartr3A}, we have
  \begin{align*}
    \bm{\phirdm^A}(O^A) & = \bm{\phirdm^A}(I^A) = \bm{\phirdm^A}(\Tr_3^A)
    = \frac{1}{4} \as
  \end{align*}

  Hence, since~$\vec C_3^A = \One_A - O^A - I^A - \Tr_3^A$, we have
  \begin{align*}
    \bm{\phirdm^A}(\vec C_3^A) = \frac{1}{4} \as
  \end{align*}

  \medskip

  For the converse, note first that if~$\phi\in\Hom^+(\cA^0,\RR)$ is such that~$\bm{\phi^A}(\vec C_3^A) = 1/4$ a.s., then
  \begin{align*}
    \phi(\vec C_3) & =
    \frac{\phi(\llbracket\vec C_3^A\rrbracket_A)}{\phi(\llbracket\One_A\rrbracket_A)}
    = \expect{\bm{\phi^A}(\vec C_3^A)}
    = \frac{1}{4}.
  \end{align*}

  Now, it is straightforward to check the following flag algebra identities.
  \begin{align*}
    R_4 & = 12\llbracket (\vec C_3^A)^2 \rrbracket_A; &
    \vec C_3 & = \frac{1}{2}R_4 + \frac{1}{4}W_4 + \frac{1}{4}L_4.
  \end{align*}

  Using the first identity along with~$\bm{\phi^A}(\vec C_3^A)=1/4$ a.s., we get
  \begin{align*}
    \phi(R_4) & = 12\phi(\llbracket (\vec C_3^A)^2 \rrbracket_A)
    = 6\expect{\bm{\phi^A}(\vec C_3^A)^2}
    = \frac{3}{8}.
  \end{align*}

  Using the second identity, we get
  \begin{align*}
    \phi(W_4 + L_4) & = 4\phi(\vec C_3) - 2\phi(R_4)
    = \frac{1}{4},
  \end{align*}
  hence~$\phi(\Tr_4+R_4) = 3/4$ (since~$\Tr_4+R_4+W_4+L_4=\One$), which is property~$P_2$ of Chung--Graham.

  Therefore~$\phi=\phirdm$.
\end{proof}

\bibliographystyle{alpha}
\bibliography{razb}

\newcommand{\etalchar}[1]{$^{#1}$}
\begin{thebibliography}{HHK{\etalchar{+}}12}

\bibitem[BR80]{BuR}
S.~A. Burr and V.~Rosta.
\newblock On the {R}amsey multiplicities of graphs---problems and recent
  results.
\newblock {\em J. Graph Theory}, 4(4):347--361, 1980.

\bibitem[CG91]{ChGr2}
F.~Chung and R.~Graham.
\newblock Quasi-random tournaments.
\newblock {\em Journal of Graph Theory}, 15(2):173--198, 1991.

\bibitem[CGW89]{CGW}
F.~Chung, R.~Graham, and R.~Wilson.
\newblock Quasi-random graphs.
\newblock {\em Combinatorica}, 9:345--362, 1989.

\bibitem[CY11]{CuY}
J.~Cummings and M.~Young.
\newblock Graphs containing triangles are not 3-common.
\newblock {\em Journal of Combinatorics}, 2:1--14, 2011.

\bibitem[DJ08]{DiJa}
P.~Diaconis and S.~Janson.
\newblock Graph limits and exchangeable random graphs.
\newblock {\em Rendiconti di Matematica, Serie VII}, 28:33--61, 2008.

\bibitem[{Erd}62]{Erd5}
P.~{Erd\H{o}s}.
\newblock On the number of complete subgraphs contained in certain graphs.
\newblock {\em Publ. Math. Inst. Hungar. Acad. Sci}, 7:459--464, 1962.

\bibitem[Gri13]{Grif}
S.~Griffiths.
\newblock Quasi-random oriented graphs.
\newblock {\em Journal of Graph Theory}, 74(2):198--209, 2013.

\bibitem[HHK{\etalchar{+}}12]{wheel}
H.~Hatami, J.~Hladky, D.~Kral, S.~Norin, and A.~Razborov.
\newblock Non-three-colorable common graphs exist.
\newblock {\em Combinatorics, Probability and Computing}, 21(5):734--742, 2012.

\bibitem[J{\v{S}}T96]{JST}
C.~Jagger, P.~{\v{S}}{\v{t}}ov{\'{\i}}{\v{c}}ek, and A.~Thomason.
\newblock Multiplicities of subgraphs.
\newblock {\em Combinatorica}, 16(1):123--141, 1996.

\bibitem[KS06]{KrS}
M.~Krivelevich and B.~Sudakov.
\newblock Pseudo-random graphs.
\newblock In {\em More Sets, Graphs and Numbers, Bolyai Society Mathematical
  Studies 15}, pages 199--262. Springer-Verlag, 2006.

\bibitem[KS13]{KaSa}
S.~Kalyanasundaram and A.~Shapira.
\newblock A note on even cycles and quasirandom tournaments.
\newblock {\em Journal of Graph Theory}, 73(3):260--266, 2013.

\bibitem[Raz07]{flag}
A.~Razborov.
\newblock Flag algebras.
\newblock {\em Journal of Symbolic Logic}, 72(4):1239--1282, 2007.

\bibitem[Raz13a]{ch}
A.~Razborov.
\newblock On the {Caccetta-Haggkvist} conjecture with forbidden subgraphs.
\newblock {\em Journal of Graph Theory}, 74(2):236--248, 2013.

\bibitem[Raz13b]{flag_ams}
A.~Razborov.
\newblock What is a {Flag Algebra?}
\newblock {\em Notices of the AMS}, 60(10):1324--1327, 2013.

\bibitem[Sid89]{Sid2}
A.~Sidorenko.
\newblock Cycles in graphs and functional inequalities.
\newblock {\em Mat. Zametki}, 46(5):72--79, 104, 1989.

\bibitem[Tho87]{Tho2}
A.~Thomason.
\newblock Pseudo-random graphs.
\newblock {\em Annals of Discrete Math.}, 33:307--331, 1987.

\bibitem[Tho89]{Tho}
A.~Thomason.
\newblock A disproof of a conjecture of {Erd\H{o}s} in {Ramsey} theory.
\newblock {\em Journal of the London Mathematical Society}, 39:246--255, 1989.

\end{thebibliography}
\end{document}